\documentclass[12pt]{article}
\usepackage{amssymb, amsmath, amsthm, amsfonts}
\usepackage{newtxtext,newtxmath}

\newtheorem{theorem}{Theorem}
\newtheorem{lemma}{Lemma}
\DeclareRobustCommand{\rchi}{{\mathpalette\irchi\relax}}
\newcommand{\irchi}[2]{\raisebox{\depth}{$#1\chi$}}

\begin{document}

\centerline{\large\bf True-pairs of Real Linear Operators and Factorization of Real Polynomials}

\vspace{6pt}
\begin{center}
Arindama Singh \\
Department of Mathematics \\
Indian Institute of Technology Madras \\
Chennai-6000036, India \\
Email: asingh@iitm.ac.in
\end{center}

\noindent {\bf Abstract}: A linear operator on a finite dimensional nonzero real vector space may not have an eigenvalue. However, corresponding to each such operator $T$, there exist a pair of real numbers $(\alpha,\beta)$ and a nonzero vector $v$ such that $[(T-\alpha I)^2+\beta^2 I](v)=0$. This is usually proved by using the Fundamental theorem of algebra and Cayley-Hamilton theorem. We construct an inductive proof of this fact without using the Fundamental theorem of Algebra. From this we deduce that a polynomial with real coefficients can be written as a product of linear factors and quadratic factors with negative discriminant. It thus gives a proof of the latter fact about polynomials with real coefficients, which does not use complex numbers.

\vspace{6pt}

\vspace{6pt}

\section{Introduction}
Let $V$ be a real vector space of dimension $n\geq 1$. Let $T:V\to V$ be a linear operator. Recall that a real number $\lambda$ is called an eigenvalue of $T$ if there exists a nonzero vector $u\in V$ such that $Tu=\lambda u$. In this case, any such vector $u$ is called an eigenvector. The polynomial $\rchi_T(t):={\rm det}(t I-T)$ is called the characteristic polynomial of $T$, and it is a monic polynomial of degree $n$ with real coefficients. Eigenvalues of $T$ are precisely the real zeroes of $\rchi_T(t)$. Thus, there are linear operators on real vector spaces having no eigenvalues. For instance, the linear operator $T:\mathbb{R}^2\to\mathbb{R}^2$ given by $T(a,b)=(-b,a)$ does not have an eigenvalue. However, we see that $(T^2+I)(a,b)=0$ for any $(a,b)\in\mathbb{R}^2$.

Due to the fundamental theorem of algebra, $\rchi_T(t)$ has $n$ number of complex zeroes counting multiplicities. Since $\rchi_T(t)$ has real coefficients, all its complex zeroes occur in conjugate pairs. That is, it can be written in the form
$$\rchi_T(t)=(t-a_1)\cdots (t-a_m) (t-b_1)(t-\overline{b}_1)\cdots (t-b_k)(t-\overline{b}_k)$$
where $a_i$ are real numbers and $b_j$ are complex numbers with nonzero imaginary parts. Further, if $\rchi_T(t)$ has no real zeroes, then $a_1,\ldots,a_m$ are absent in the above product. Writing each $b_j=\alpha_j+i\beta_j$ for real numbers $\alpha_j,\beta_j$ with $\beta_j\neq 0$, we see that $(t-b_j)(t-\overline{b}_j)=(t-\alpha_j)^2+\beta_j^2$. Thus,
\begin{equation}\label{facto}
\rchi_T(t)=(t-a_1)\cdots (t-a_m) \big((t-\alpha_1)^2+\beta_1^2\big)\cdots \big((t-\alpha_k)^2+\beta_k^2\big).
\end{equation}
Due to Cayley-Hamilton theorem, $T$ satisfies its characteristic polynomial. That is,
$$(T-a_1 I)\cdots (T-a_m I) \big((T-\alpha_1 I)^2+\beta_1^2 I\big)\cdots \big((T-\alpha_k I)^2+\beta_k^2 I\big)=0.$$
Multiplying $(T-a_1 I)\cdots (T-a_m I)$ with this and rewriting $(T-a_i I)^2$ as $(T-\alpha_i I)^2+0^2 I$, we have
$$\big((T-a_1 I)^2+0^2I\big)\cdots \big((T-a_m I)^2+0^2 I\big) \big((T-\alpha_1 I)^2+\beta_1^2 I\big)\cdots \big((T-\alpha_k I)^2+\beta_k^2 I\big)=0.$$
Let $v_k$ be a nonzero vector. Write $v_{k-1}=:\big((T-\alpha_k I)^2+\beta_k^2 I\big)(v_k)$. From the above equation it follows that either $v_{k-1}=0$ or
$$\big((T-a_1 I)^2+0^2 I\big)\cdots \big((T-a_m I)^2+0^2 I\big) \big((T-\alpha_1 I)^2+\beta_1^2 I\big)\cdots \big((T-\alpha_{k-1} I)^2+\beta_{k-1}^2 I\big)(v_{k-1})=0.$$
Proceeding inductively we see that there exist $\alpha,\beta\in\mathbb{R}$ and a nozero vector $v\in V$ such that
$$[(T-\alpha I)^2+\beta^2 I](v)=0.$$
We wish to construct a proof of this fact which does not use the complex numbers, the fundamental theorem of algebra, and/or Cayley-Hamilton theorem. For this purpose, we use the following facts, whose proofs do not depend upon complex numbers, the fundamental theorem of algebra or Cayley-Hamilton theorem. By a real polynomial we mean a polynomial in a single variable $t$ with real coefficients.

\begin{description}
\item[Fact 1:] Any real polynomial of odd degree has a real zero.
\item[Fact 2:] Any real polynomial of degree four can be expressed as a product of two real polynomials of degree two each.
\item[Fact 3:] For each monic real polynomial $p(t)$ of degree $n$, there exists a matrix $A$ of order $n$, called the companion matrix of $p(t)$, such that the characteristic polynomial ${\rm det}(tI-A)$ of $A$ is $p(t)$.
\end{description}

Fact 1 follows from the Intermediate value property of continuous real valued functions. There are many methods to express a real polynomial of degree four as a product of two quadratic factors mentioned in Fact 2, which do not use complex numbers. For two such methods by Ferrari and Descartes, see  articles 12-14 of Chapter XII in \cite{bern}. Fact 3 is discussed in almost all standard books on Linear Algebra; for instance, see Sect. 5.2 in \cite{nair}.

Later, we will deduce the existence of a factorization of real polynomials in form (\ref{facto}) by using the division algorithm and Cayley-Hamilton theorem.

\section{Preliminary results}
In this section, we introduce some terminology and prove some results, which will lead to our main result.

Let $V$ be a finite dimensional nonzero real vector space. Let $T:V\to V$ be a linear operator on $V$. We say that $T$ {\em has a true-pair vector} if there exist real numbers $\alpha,\beta$ and a nonzero vector $v\in V$ such that
$$[(T-\alpha I)^2+\beta^2 I](v)=0.$$
In such a case, we say that the pair of real numbers $(\alpha,\beta)$ is a {\em true-pair} of $T$ and $v$ is an associated {\em true-pair vector} of $T$.

If $T$ has an eigenvalue $\lambda$ with an associated eigenvector $u$, then $(T-\lambda I)(u)=0$. It implies that $[(T-\lambda I)^2+0^2 I](u)=0$; that is, $T$ has a true-pair vector, namely $u$.

We write
\begin{eqnarray*}
N(T) &=& \{x\in V: T(x)=0\},~{\rm the~null~space~of~} T\,; \\
R(T) &=& \{T(x): x\in V\},~{\rm the~range~space~of~} T.
\end{eqnarray*}

We say that two linear operators $S$ and $T$ on $V$ have a {\em common true-pair vector} if there exists a nonzero vector $v\in V$ and real numbers $\alpha,\beta,\gamma,\delta$ such that
$$[(S-\alpha I)^2+\beta^2 I](v)=0\quad\mbox{and}\quad [(T-\gamma I)^2+\delta^2 I](v)=0.$$
In such a case, $v$ is  said to be a {\em common true-pair vector} of $S$ and $T$.

Further, two linear operators $S$ and $T$ on $V$ are said to be {\em commuting operators} if $S(T(x))=T(S(x))$ for each $x\in V$.

\begin{lemma}\label{lem3}
Let $S$ and $T$ be two commuting operators on a finite dimensional nonzero real vector space $V$. Let $(\alpha,\beta)$ be a true-pair of $S$. Then the restrictions of $S$ and $T$ to $N\big((S-\alpha I)^2+\beta^2 I\big)$ are commuting operators, and the restrictions of $S$ and $T$ to $R\big((S-\alpha I)^2+\beta^2 I\big)$ are commuting operators.
\end{lemma}

\begin{proof}
%
%
%
Since $(S-\alpha I)^2+\beta^2 I$ is a linear operator on $V$, write
$${\cal N} = N\big((S-\alpha I)^2+\beta^2 I\big),\quad {\cal R} = R\big((S-\alpha I)^2+\beta^2 I\big).$$
These are subspaces of $V$. Let $x\in {\cal N}$. Then $[(S-\alpha I)^2+\beta^2 I](x)=0$. Now,
\begin{eqnarray*}
[(S-\alpha I)^2+\beta^2 I]S(x) &=& S[(S-\alpha I)^2+\beta^2 I](x)=S(0)=0, \\ ~
[(S-\alpha I)^2+\beta^2 I]T(x) &= & T[(S-\alpha I)^2+\beta^2 I](x)=T(0)=0.
\end{eqnarray*}
That is, $S(x)\in {\cal N}$ and $T(x)\in {\cal N}$. Hence, ${\cal N}$ is invariant under both $S$ and $T$.

Next, let $y\in {\cal R}$. There exists $x\in V$ such that $y=[(S-\alpha  I)^2+\beta^2 I](x)$. Then,
\begin{eqnarray*}
S(y) &=& S[(S-\alpha I)^2+\beta^2 I](x)=[(S-\alpha I)^2+\beta^2 I]S(x)\in {\cal R}, \\
T(y) &=& T[(S-\alpha I)^2+\beta^2 I](x) = [(S-\alpha I)^2+\beta^2 I]T(x)\in {\cal R}.
\end{eqnarray*}
That is, ${\cal R}$ is invariant under both $S$ and $T$.

Now that the subspaces ${\cal N}$ and ${\cal R}$ of $V$ are invariant under both $S$ and $T$, the conclusions follow.
 \end{proof}

In \cite{derk}, it has been shown that a finite number of commuting operators on a finite dimensional nonzero complex vector space have a common eigenvector. We apply a similar technique for obtaining information about true-pair vectors without using complex numbers.

\begin{lemma}\label{lem2}
Let $k\in\mathbb{N} $. Let $V$ be a finite dimensional nonzero real vector space, where $2^k$ does not divide $\dim(V)$.   If each linear operator on $V$ has a true-pair vector, then any two commuting operators on $V$ have a common true-pair vector.
\end{lemma}
\begin{proof}
We use induction on the dimension of $V$. Suppose $\dim(V)=1$. Then, $k=1$. Let $\{v\}$ be a basis for $V$. Suppose $S$ and $T$ are linear operators on $V$. Then, $S(v)\in V$ implies that there exists $\alpha\in\mathbb{R}$ such that $S(v)=\alpha v$. Similarly, $T(v)=\gamma v$ for some $\gamma\in\mathbb{R}$. Now,
$$[(S-\alpha I)^2+0\cdot I](v)=0=[(T-\gamma I)^2+0\cdot I](v).$$
So, $v$ is a common true-pair vector of $S$ and $T$.

Assume the induction hypothesis that if $U$ is a nonzero real vector space of dimension less than $n$ and $2^k$ does not divide $\dim(U)$, then each linear operator on $U$ has a true-pair vector implies that any two commuting operators on $U$ have a true-pair vector. Let $V$ be a nonzero real vector space of dimension $n$, where $2^k$ does not divide $n$. Assume that each linear operator on $V$ has  a true-pair vector. Let $S$ and $T$ be two commuting operators on $V$. We need to show that $S$ and $T$ have a common true-pair vector.

Now that $S$ and $T$ have true-pair vectors, there exist $\alpha,\beta,\gamma,\delta\in\mathbb{R}$ and nonzero vectors $u,v\in V$ such that
$$[(S-\alpha I)^2+\beta^2 I](u)=0,\quad [(T-\gamma I)^2+\delta^2 I](v)=0.$$

Write ${\cal N}=N\big((S-\alpha I)^2+\beta^2 I\big)$. Since $u$ is a nonzero vector in ${\cal N}$, $\dim {\cal N}\geq 1$. Further, for each $x\in {\cal N}$, $[(S-\alpha I)^2+\beta^2 I](x)=0$.

If ${\cal N}=V$, then $v\in {\cal N}$ so that $[(S-\alpha I)^2+\beta^2 I](v)=0$. Now, $S$ and $T$ have a common true-pair vector, namely, $v$.

So, let ${\cal N}$ be  a proper subspace of $V$; that is, $\dim({\cal N})<n$. By Lemma~\ref{lem3}, the restrictions of $S$ and $T$ to ${\cal N}$ are commuting operators.

If $2^k$ does not divide $\dim({\cal N})$, then by the induction hypothesis, the restriction operators of $S$ and $T$ to ${\cal N}$ have a true-pair vector $y\in {\cal N}$. Then, $y$ is a common true-pair vector of $S$ and $T$ as linear operators on $V$.

So, assume that $2^k$ divides $\dim({\cal N})$. Write ${\cal R}=R\big((S-\alpha I)^2+\beta^2 I\big)$. By Lemma~\ref{lem3}, the restrictions of $S$ and $T$ to ${\cal R}$ are commuting operators. Due to the Rank-nullity theorem,  $\dim({\cal N})+\dim({\cal R})=n$. As $2^k$ divides $\dim({\cal N})$ and $2^k$ does not divide $n$, it follows that $2^k$ does not divide $\dim({\cal R})$. As $\dim({\cal R})<n$, by the induction hypothesis, the restrictions of $S$ and $T$ to ${\cal R}$ have a common true-pair vector $y$. Then, $y$ is a common true-pair vector of $S$ and $T$ as linear operators on $V$.
\end{proof}

We slightly enlarge our vocabulary.  Let $m\in\mathbb{N} \cup\{0\}$. By an {\em operator of evenness} $m$, we mean a linear operator on a finite dimensional nonzero real vector space $V$, where the maximum power of $2$ that divides $\dim(V)$ is $2^m$.   Notice that an  operator of evenness $0$ is simply a linear operator on an odd dimensional real vector space.

\begin{lemma}\label{lem4}
Let $m\in\mathbb{N} \cup\{0\}$. If each operator of evenness $m$ has a true-pair vector, then each operator of evenness $m+1$ has a true-pair vector.
\end{lemma}
\begin{proof}
Assume that each operator of evenness $m$ has a true-pair vector. Let $T$ be an operator of evenness $m+1$. That is,  $T:V\to V$ is a linear operator, where $V$ is a finite dimensional nonzero real vector space with $\dim(V)=n$ and $2^{m+1}$ is the maximum power of $2$ that divides $n$. Fix an ordered basis for $V$. Let $A$ be the matrix representation of $T$ with respect to this ordered basis. Then, $A$ is an $n\times n$ matrix with real entries.

Now, $A=PTP^{-1}$, where $P$ is the canonical basis isomorphism from $V$ to $\mathbb{R}^{n\times 1}$. Then, for any $\alpha,\beta\in\mathbb{R}$,
$$
P[(T-\alpha I)^2+\beta^2 I]P^{-1} = P(T-\alpha I)P^{-1}P(T-\alpha I)P^{-1}+\beta^2 I \\
= (A-\alpha I)^2+\beta^2 I.
$$
Thus, if $u$ is a true-pair vector of $A$, then there exist a pair of real numbers $(\alpha,\beta)$ such that  $\big((A-\alpha I)^2+\beta^2 I\big)u=0$. It gives $P[(T-\alpha I)^2+\beta^2 I]\big(P^{-1}(u)\big)=0$, which implies
$$[(T-\alpha I)^2+\beta^2 I]\big(P^{-1}(u)\big)=0.$$
That is, $P^{-1}(u)$ is a true-pair vector of $T$. Hence, it is enough to prove that $A$ has a true-pair vector.

Consider $S_n=\{X\in\mathbb{R}^{n\times n}: X^t=X\}$, the set of all real symmetric matrices of order $n$. Then, $S_n$ is a real vector space of dimension $n(n+1)/2$. Define
$$L_1(X)=AX+XA^t,\quad L_2(X)=AXA^t\quad\mbox{for~} X\in S_n.$$
For $X,Y\in S_n$ and $b\in\mathbb{R}$, we have
\begin{eqnarray*}
\big(L_1(X)\big)^t&=&(AX+XA^t)^t=X^tA^t+AX^t=XA^t+AX=L_1(X), \\
\big(L_2(X)\big)^t&=&(AXA^t)^t=AX^tA^t=AXA^t=L_2(X), \\
L_1(bX+Y)&=& A(bX+Y)A^t= bAXA^t+AYA^t=bL_1(X)+L_1(Y), \\
L_2(bX+Y)&=& A(bX+Y)A^t = bAXA^t+AYA^t=bL_2(X)+L_2(Y), \\
L_1(L_2(X))&=&L_1(AXA^t)=A(AXA^t)+(AXA^t)A^t=A^2XA^t+AX(A^t)^2, \\
L_2(L_1(X))&=& L_2(AX+XA^t) = A(AX+XA^t)A^t=A^2XA^t+AX(A^t)^2.
\end{eqnarray*}
So, $L_1$ and $L_2$ are commuting operators on $S_n$.

Notice that  $n=2^{m+1} k$ for an odd integer $k$, and $m\geq 0$. Thus, $n+1$ is odd so that $n(n+1)/2=2^{m}(n+1)k$, where $(n+1)k$ is an odd integer. Then, every linear operator on $S_n$ is of evenness $m$. By assumption, every linear operator on $S_n$ has a true-pair vector.  Due to Lemma~\ref{lem2}, $L_1$ and $L_2$ have a common true-pair vector. So, let $\alpha,\beta,\gamma,\delta\in\mathbb{R}$ and $B\in S_n$ be a nonzero matrix such that
$$[(L_1-\alpha I)^2+\beta^2 I](B)=0,\quad [(L_2-\gamma I)^2+\delta^2 I](B)=0.$$
Then, $-\beta^2\delta^2B=\delta^2(L_1-\alpha I)^2(B)=\beta^2(L_2-\gamma I)^2(B)$. It gives
\begin{equation}\label{eq1}
[\delta(L_1-\alpha I)+\beta(L_2-\gamma I)][\delta(L_1-\alpha I)-\beta(L_2-\gamma I)](B)=0.
\end{equation}

\noindent{\em Case 1:} Suppose  $[\delta(L_1-\alpha I)-\beta(L_2-\gamma I)](B)=0$. Then,
$$\delta(AB+BA^t-\alpha B)-\beta(ABA^t-\gamma B)=0.$$
It implies $(\delta I-\beta A)BA^t=\big(-\delta A+(\alpha\delta-\beta\gamma)I\big)B.$ Using this, we obtain the following:
\begin{eqnarray*}
(\delta I-\beta A)L_1(B)&=& (\delta I-\beta A)AB+(\delta I-\beta A)BA^t \\
&=& (\delta I-\beta A)AB+ [-\delta AB+(\alpha\delta-\beta\gamma)B] \\
&=& -\beta A^2B+(\alpha\delta-\beta\gamma)B \\
&=& \big(-\beta A^2+(\alpha\delta-\beta\gamma) I\big)B. \\
(\delta I -\beta A)^2BA^tA^t &=& (\delta I-\beta A)\big(-\delta A+(\alpha\delta-\beta\gamma)I\big)BA^t \\
&=& \big(-\delta A+(\alpha\delta-\beta\gamma)I\big)(\delta I-\beta A) BA^t \\
&=& \big(-\delta A+(\alpha\delta-\beta\gamma)I\big)^2B
\end{eqnarray*}
\begin{eqnarray*}
(\delta I-\beta A)^2L_1^2(B)&=& (\delta I-\beta A)^2L_1\big(AB+BA^t\big) \\
&=& (\delta I-\beta A)^2\big(A(AB+BA^t)+(AB+BA^t)A^t\big) \\
&=& (\delta I-\beta A)^2(A^2B+2ABA^t+BA^t A^t) \\
&=& (\delta I-\beta A)^2A^2B+2(\delta I-\beta A)^2A BA^t+(\delta I-\beta A)^2 BA^t A^t \\
&=& (\delta I-\beta A)^2A^2B+2(\delta I-\beta A)A (\delta I-\beta A)BA^t \\ && +(\delta I-\beta A)^2 BA^t A^t \\
&=& (\delta I-\beta A)^2A^2B+2(\delta I-\beta A)A\big(-\delta A +(\alpha\delta-\beta\gamma)\big)B \\
&& + \big(-\delta A+(\alpha\delta-\beta\gamma)I\big)^2B\\
&=& \big( (\delta I-\beta A)A+(-\delta A+ (\alpha\delta-\beta\gamma)I\big)^2B \\
&=& \big(-\beta A^2+(\alpha\delta-\beta\gamma)I\big)^2B.
\end{eqnarray*}
Then,
\begin{eqnarray*}
0 &=& (\delta I-\beta A)^2 \big((L_1-\alpha I)^2+\beta^2 I\big)(B) \\
&=& (\delta I-\beta A)^2L_1^2(B)-2\alpha(\delta I-\beta A)^2L_1(B)+(\delta I-\beta A)^2(\alpha^2+\beta^2)B \\
&=& \big(-\beta A^2+(\alpha\delta-\beta\gamma)I\big)^2B-2\alpha(\delta I-\beta A)\big(-\beta A^2+ (\alpha\delta-\beta\gamma)I\big) B\\
&& +(\delta I-\beta A)^2(\alpha^2+\beta^2)B \\
&=& \big((-A^2+\alpha A-\gamma I)^2+(\delta I-\beta A)^2\big)B.
\end{eqnarray*}
Due to Fact~2, there exist $a,b,c,d\in\mathbb{R}$ such that
$$(-t^2+\alpha t-\gamma)^2+(\delta-\beta t)^2=(t^2-at+b)(t^2-ct+d).$$
Therefore,
\begin{equation}\label{eq2}
(A^2-aA+b I)(A^2-cA+d I)B=0.
\end{equation}

\noindent {\em Case 1A:} Suppose $(A^2-cA+d I)B=0$. As $B\neq 0$, let $x$ be a nonzero column of $B$. Then, $(A^2-cA+d I)x=0$.

If $c^2-4d \geq 0$, then write $\gamma=\big(c+\sqrt{c^2-4d}\,\big)/2$ and $\delta=\big(c-\sqrt{c^2-4d}\,\big)/2.$ Now, $\gamma,\delta\in\mathbb{R}$ and
$$(A-\gamma I)(A-\delta I)x=(A^2-c A+d I)x=0.$$

If $(A-\delta I)x=0$, then $[(A-\delta I)^2+0\cdot I]x=(A-\delta I)[(A-\gamma I)(A-\delta I)x]=0$ shows that $x$ is a true-pair vector of $A$.

If $(A-\delta I)x\neq 0$, then $[(A-\gamma I)^2+0\cdot I](A-\delta I)x=0$ shows that $(A-\delta I)x$ is a true-pair vector of $A$.

If $c^2-4d <0$, then write $\gamma=c/2$ and $\delta=\big(\sqrt{4d-c^2}\,\big)/2.$ Now, $\gamma,\delta\in\mathbb{R}$ and
$$[(A-\gamma I)^2 + \delta^2 I]x=(A^2-c A+d I)x=0.$$
Thus, $x$ is a true-pair vector of $A$.

\vspace{4pt}
\noindent {\em Case 1B:} Suppose $C:=(A^2-cA+d I)B\neq 0$. Equation~\ref{eq2} gives $(A^2-aA+b I)C=0$. This case reduces to Case 1A with $a, b, C$ in place of $c, d, B$, respectively.

\vspace{4pt}
\noindent{\em Case 2:} Suppose  $D:=[\delta(L_1-\alpha I)-\beta(L_2-\gamma I)](B)\neq 0$. Equation~\ref{eq1} yields
$$[\delta(L_1-\alpha I)+\beta(L_2-\gamma I)]D=0.$$
This case is reduced to Case~1 with $-\beta$ in place of $\beta$ and $D$ in place of $B$.
\end{proof}

\section{Main results}
In this section we prove our main result and then derive the existence of the intended factorization of a polynomial wih real coefficients.

\begin{theorem}\label{thm1}
Every linear operator on a finite dimensional nonzero real vector space has a true-pair vector.
\end{theorem}
\begin{proof}
We use induction on the evenness of a linear operator. Let $V$ be a finite dimensional real vector space. In the basis step, when a linear operator $T:V\to V$ is of evenness $0$, $V$ is an odd dimensional real vector space. The characteristic polynomial of $T$ is of odd degree. By Fact~1, it has a real zero, say, $\lambda$. Then, $\lambda$ is an eigenvalue of $T$ with an associated eigenvector $v$. As remarked earlier, $v$ is a true-pair vector of $T$.

Assume the induction hypothesis that each operator of evenness $m$ has a true-pair vector. Let $T:V\to V$ be a linear operator with evenness $m+1$. By Lemma~\ref{lem4}, $T$ has a true-pair vector.
\end{proof}


As a corollary to Theorem~\ref{thm1}, we obtain the following result about real polynomials.

\begin{theorem}\label{thm2}
Every non-constant real polynomial in a single variable $t$ has either a real linear factor in the form $t-\alpha$ or a real quadratic factor in the form $(t-\beta)^2+\gamma^2$ for real numbers $\alpha,\beta$ and $\gamma$.
\end{theorem}
\begin{proof}
Without loss in generality, consider monic polynomials with degree at least $3$. So, let
$$p(t)=a_1+a_2 t+a_3 t^2+\cdots + a_nt^{n-1}+t_n,$$
where $n\in\mathbb{N} ,~n\geq 3$ and $a_1,a_2,\ldots,a_n\in\mathbb{R}$. In view of Fact~3, let $A$ be a matrix of order $n$ such that $\det(t I-A)=p(t)$. By Theorem~\ref{thm1}, there exist $\beta,\gamma\in\mathbb{R}$ and a nonzero vector $v\in\mathbb{R}^{n\times 1}$ such that $[(A-\beta I)^2 + \gamma^2 I]v=0$.  Let $r(t)=(t-\beta)^2+\gamma^2$. Now, $r(A)v=0$. By the Division algorithm,

\vspace{6pt}
\centerline{either $~p(t)=q(t)r(t)~$ or $~p(t)=q(t)r(t)+(at-b)~$}

\vspace{6pt}
\noindent for some real polynomial $q(t)$ of degree $n-2$ and some $a,b\in\mathbb{R}$. In the former case, we are through. In the latter case, $p(A)v=q(A)r(A)v+(aA-bI)v$. By Cayley-Hamilton theorem, $p(A)=0$. So, $(aA-bI)v=0$.

If $a=0$, then $bv=0$; and $v\neq 0$ implies $b=0$. Hence, $p(t)=q(t)r(t)$.

If $a\neq 0$, then $b/a$ is an eigenvalue of $A$. Thus, $p(b/a)=0$. Then, $p(t)=(t-b/a)q_1(t)$ for some real polynomial $q_1(t)$ of degree $n-1$.
\end{proof}

Observe that a quadratic factor $(t-\alpha)^2+\beta^2$ of $p(t)$ includes two cases. If $\beta=0$, then the real linear polynomial $t-\alpha$ divides $p(t)$. And, if $\beta\neq 0$, then the discriminant of $(t-\alpha)^2+\beta^2$ is $4\alpha^2-4(\alpha^2+\beta^2)=-4\beta^2<0$; so that a real quadratic polynomial with negative discriminant, namely, $(t-\alpha)^2+\beta^2$ divides $p(t)$. Thus, applying Theorem~\ref{thm2} repeatedly, we obtain the required factorization proving the following statement.

\begin{theorem}\label{thm3}
Every real polynomial $p(t)$ of degree $n\geq 1$ can be factorized as
$$p(t)=a\cdot L_1(t)\cdots L_m(t)\cdot Q_1(t)\cdots Q_k(t)$$
where $a\in\mathbb{R}$, $L_i(t)$ are linear real polynomials and $Q_j(t)$ are quadratic real polynomials with negative discriminant.
\end{theorem}

Notice that the factorization in Theorem~\ref{thm3} is unique since $\mathbb{R}[t]$ is a unique factorization domain. Further, in Theorem~\ref{thm3}, $n=m+2k$ for $n,m,k\in\mathbb{N}\cup\{0\}$.

\section{Conclusions}
As shown in Section~1, use of complex numbers, the Fundamental theorem of algebra  and Cayley-Hamilton theorem lead to a simple proof of Theorem~\ref{thm3}. However, the statement of this result does not involve complex numbers. Thus, one would expect to obtain a proof which does not use complex numbers. This problem has been mentioned in \cite{srini} as the open problem 2. We have constructed such a proof by using Linear algebra. During this construction we have introduced the notions of a true-pair vector and the evenness of a linear operator on finite dimensional real vector spaces. These notions served as abbreviations only; they helped us to present the proof in a comprehensible manner. We remark that the results discussed here neither use nor prove the fundamental theorem of algebra. It is pertinent to note that perhaps the author in \cite{srini} is asking for a proof of the factorization of real polynomials which uses the techniques of real analysis. In that sense, the problem is still open.

%
%
%


\begin{thebibliography}{99}


\bibitem{bern}
S. Bernard, S, J. M. Child,  \emph{Higher Algebra}, Macmillan \& Co Ltd, London, 1967.

\bibitem{derk}
H. Derksen, \emph{The fundamental theorem of algebra and linear algebra}, Amer. Math. Monthly  110 (2003), pp. 620--623.

\bibitem{nair}
M. T. Nair, A. Singh, \emph{Linear Algebra}, Springer, 2018.

\bibitem{srini}
V. K. Srinivasan, \emph{Open problems in analysis}, Int. J. Math. Educ. Sci. Technol. 28 (1997) pp. 117--121.

\end{thebibliography}
\end{document}